\documentclass[11pt]{article}
\usepackage[margin=1 in]{geometry}
\usepackage{xcolor}
\usepackage{relsize}
\usepackage{fullpage}
\usepackage{amsfonts}
\usepackage{graphicx}
\usepackage{subfig}
\usepackage{amsmath,cite}
\usepackage{accents}
\usepackage{latexsym,array,multirow,geometry}
\usepackage{tikz,multicol,enumitem,pgfplots}
\usetikzlibrary{positioning,chains,fit,shapes,calc}
\usetikzlibrary{patterns}
\usepgfplotslibrary{fillbetween}
\usepackage[colorlinks=true,allcolors=blue]{hyperref}
\usepackage[nameinlink,noabbrev,capitalize]{cleveref}
\usepackage{amssymb}
\usepackage{amsthm}
\usepackage{stmaryrd}
\usepackage{tikz-network}
\usepackage{mathtools}

\usepackage[latin1]{inputenc}
\usepackage{tikz}
\usetikzlibrary{3d}
\usepackage{tikz-3dplot}
\usetikzlibrary{shapes,arrows}
\usepackage{adjustbox}

\usepackage{amsmath,amssymb}
\usepackage{amsthm}
\usepackage{mathtools}
\usepackage{tikz-cd}
\usepackage{mathrsfs}
\usepackage{fullpage}
\usepackage{graphicx}
\usepackage{tikz}
\usepackage{dsfont}
\usepackage[english]{babel}
\usepackage{fancyhdr}
\usepackage{wrapfig}
\usepackage{color}
\usepackage{colortbl}
\usepackage{enumitem}
\usepackage{hyperref}
\usepackage{calligra}
\usepackage{multirow, multicol}
\usepackage[linesnumbered,ruled,vlined]{algorithm2e}
\usepackage{caption}
\usepackage{subcaption}
\usepackage{mathrsfs}
\usepackage[authoryear]{natbib}

\newcommand{\cS}{{\mathcal S}}

\def\Prob{{\mathbb P}}

\def\col{{\rm col}}
\def\Event{{\mathcal E}}
\def\Exp{{\mathbb E}}

\def\Var{{\rm Var}}

\newcommand*{\N}{\mathbb{N}}

\newcommand*{\R}{\mathbb{R}}

\newtheorem{theorem}{Theorem}

\newtheorem{prop}[theorem]{Proposition}

\theoremstyle{definition}

\newtheorem{assumption}[theorem]{Assumption}
\newtheorem*{assumption*}{Assumptions}

\newtheorem*{remark}{Remark}

\title{
\textbf{A decomposition algorithm for two-stage stochastic programs with approximate rotational invariance} 
}

\author{Marzieh Bakhshi\footnote{Department of Industrial and System Engineering, University of Tennessee, Knoxville. Email: mbakhsh1@vols.utk.edu}\; and Konstantin Tikhomirov\footnote{Department of Mathematical Sciences, Carnegie Mellon University. Email: ktikhomi@andrew.cmu.edu}}

\date{\today}
\begin{document}

\maketitle

\begin{abstract}
We propose an algorithm of approximating the optimal objective value of a two-stage stochastic program under an assumption of {\it approximate rotational invariance} of the technology matrix, and compare the method with the L-shaped decomposition.
\end{abstract}
	
\pagenumbering{roman}
\pagenumbering{arabic}


\section{Introduction}
Consider the general formulation of a two-stage stochastic program 
\begin{equation} \label{sp}
\begin{split}
    &\text{maximize} \  \ \langle c,x \rangle + \Exp_{\boldsymbol{\xi}} [Q(x,\boldsymbol{\xi})] \\
    &\text{subject to }\ Ax \leq b \\
    & x \geq 0.
\end{split}
\end{equation}
where for each realization $\xi$ of a random vector $\boldsymbol{\xi}$, $Q(x,\xi)$ is the optimal objective value of the recourse problem
\begin{equation} \label{sp_recourse}
\begin{split}
    &\text{maximize} \  \ \langle q(\xi),y(\xi) \rangle \\
    &\text{subject to }\ T(\xi)x +W(\xi)y(\xi) \leq h(\xi) \\
    & y(\xi) \geq 0.
\end{split}
\end{equation}
In program (\ref{sp}), $c, x \in \R^{n_1}$ are the first stage cost and variable vectors,  $A \in \R^{m_1 \times n_1}$ encodes first stage constraints, and $b \in \R^{m_1}$ denotes the right hand side of the constraints.
In the recourse problem (\ref{sp_recourse}), $q(\xi), y(\xi) \in \R^{n_2}$ are the second stage cost and variable vectors, and the technology matrix $T(\xi) \in \R^{m_2 \times n_1}$  along with the recourse matrix $W(\xi) \in \R^{m_2 \times n_2}$ determine the quantitative relationship between the first and the second stage variables. The vector $h(\xi) \in \R^{m_2}$ denotes the right hand side of the second stage constraints.  
\smallskip

Combining the recourse problem together with the first stage constraints, and  considering all realizations of $\boldsymbol{\xi}$, we obtain the so-called 
extensive form or deterministic equivalent of problem \eqref{sp}:  
\begin{equation} \label{ef}
\begin{split}
    \text{maximize}\  & \langle c,x \rangle + \sum_{\xi\in\cS} \Prob(\xi)  \langle q(\xi), y(\xi) \rangle  \\
    \text{subject to } & Ax \leq b \\
    & T(\xi) x +W(\xi) y(\xi) \leq h(\xi) \\
    & x \geq 0 \\
    & y(\xi) \geq 0.
\end{split}
\end{equation}
The two-stage problem (\ref{sp}) can be approached by considering the equivalent deterministic problem (\ref{ef}) directly using optimization solvers.
However, problem (\ref{ef}) can be computationally expensive when the number of scenarios (i.e distinct realizations of $\boldsymbol{\xi}$) is large.

\smallskip

To resolve the issue of computational inefficiency, various decomposition algorithms have been proposed and employed in practice.
The L-shaped method ~\citep{van1969shaped} is the most widely used decomposition algorithm for solving two-stage stochastic programs. The method employs Benders' decomposition ~\citep{MR147303, laporte1993integer}. 
The algorithm starts by defining a master problem with only the first-stage decision variables. After the master problem is solved, the first-stage variables are fixed at an optimal solution of the master problem. Then the second stage problem can be decomposed into $|\cS|$ subproblems where $\cS$ denotes the set of scenarios. The subproblems can be solved in parallel, and valid inequalities can be derived and added to master problem. The master problem is then solved again and the algorithm iterates. 
The algorithm is known to converge in a finite number of steps when the second-stage decision variables are continuous and the number of scenarios is finite.

A related decomposition technique is the Lagrangian method ~\citep{guignard2003lagrangean} where non-anticipa\-tivity constraints are typically enforced through the use of Lagrangian multipliers (dual variables) associated with those constraints. The Lagrangian relaxation of the problem incorporates those constraints along with the original constraints and the objective function. The dual variables associated with the non-anticipativity constraints capture the price of violating those constraints, and they are updated iteratively as part of the decomposition process.

\smallskip

Apart from Benders' decomposition and its variations, a substantial body of literature has emerged presenting a variety of decomposition techniques for two-stage stochastic programs.
In particular, one can mention stochastic decomposition ~\citep{higle1991statistical}, subgradient decomposition ~\citep{sen1993subgradient}, level decomposition ~\citep{zverovich2012computational}, and partition-based algorithms ~\citep{song2015adaptive, van2018adaptive}. 





\bigskip

%

Existing decomposition algorithms are typically designed for generic technology matrices, which define the relationship between the first and the second stage variables. However, what happens if the technology matrix has a special structure? Can one exploit that structure for efficient decomposition of stochastic programs? 
In the present work, we focus on two-stage stochastic programs where the technology matrices satisfy an {\it approximate rotational invariance} property introduced below. For such programs, we propose a method which is based on decoupling the first and the second stages.

\section{Approximation of the optimal objective: theoretical guarantees}

In this section, we consider a rigorous argument which shows that under certain property of the recourse problems which we call {\it an approximate rotational invariance}, the optimal objective value of the two-stage stochastic problem can be approximated by decoupling the first and second stages.
In what follows, denote by $\tilde x$ an optimal solution
to the first stage problem
\begin{equation} \label{nd1}
\begin{split}
    \text{maximize}\  & \langle c,x \rangle  \\
    \text{subject to } & Ax \leq b \\
    & x \geq 0.
\end{split}
\end{equation}
Let us describe precisely the assumptions we will employ.

\medskip

\begin{assumption*}\label{NRAsmpt}
Let $\varepsilon$ be a small positive parameter. We assume all of the following.
\begin{itemize}
\item[(a)] {\it{}First stage solutions:} The first stage solution set is feasible and bounded, and for each $\tau\geq 0$, the {\it convex} quadratic
optimization problem
\begin{equation}\label{decsptau1}
\begin{split}
    &\text{maximize} \  \ \langle c,x_\tau \rangle \\
    &\text{subject to }\ Ax_\tau \leq b \\
    & \|x_\tau\|_2 \leq \tau \\
    & x_\tau \geq 0
\end{split}
\end{equation}
is feasible and has a unique solution;
\item[(b)] {\it{}Feasibility and boundedness of the recourse problem:} For every realization of $\xi$ and for every
choice of a vector
$x\in\R^{n_1}$ with $\|x\|_2\leq
\max\{\|x'\|_2:\;Ax'\leq b,\,x'\geq 0\}$, the recourse problem (\ref{sp_recourse})
is feasible and bounded;
\item[(c)] {\it{}Approximate rotational invariance:}
For every
choice of a vector
$x\in\R^{n_1}$ with $\|x\|_2\leq
\max\{\|x'\|_2:\;Ax'\leq b,\,x'\geq 0\}$,
$$
\big|\Exp_{\boldsymbol{\xi}}\,Q(x,\boldsymbol{\xi})-
\Exp_{\boldsymbol{\xi}}\,Q((\|x\|_2,0,\dots,0)^\top,\boldsymbol{\xi})\big|
\leq \varepsilon.
$$
\item[(d)] {\it{}Approximate monotonicity of $\Exp_{\boldsymbol{\xi}} Q((\cdot,0,\dots,0)^\top, \boldsymbol{\xi})$
on an interval}:
The mapping
$$\tau\longrightarrow \Exp_{\boldsymbol{\xi}}[Q((\tau,0,\dots,0)^\top, \boldsymbol{\xi})]$$
satisfies
$$
\Exp_{\boldsymbol{\xi}}[Q((\tau_1,0,\dots,0)^\top, \boldsymbol{\xi})]
\geq \Exp_{\boldsymbol{\xi}}[Q((\tau_2,0,\dots,0)^\top, \boldsymbol{\xi})]
-\varepsilon,
$$
whenever
$\|\tilde x\|_2\leq
\tau_1\leq\tau_2\leq
\max\{\|x\|_2:\;Ax\leq b,\,x\geq 0\}$.
\end{itemize}
\end{assumption*}
Observe that the assumption (b) above is standard, and that the assumption (a) can be realized for any bounded feasible first-stage problem by adding a small random perturbation to the cost vector $c$. Thus, the main conditions above are properties (c) and (d). It is not clear apriori if property (c) is satisfied 
for any stochastic program except for the degenerate case where the technology matrix is zero. However, there is a strong indication that the approximate rotational invariance holds for recourse problems with matrices $T(\xi)$ which are realizations of {\it random Gaussian matrices} (see a remark after the proof below).

\smallskip

We have the following result:
\begin{prop}
Assume that conditions (a)--(d) above
are satisfied for some choice of $\varepsilon$.
Consider the quantity
$$
\hat z^*:=\max\big\{\langle c,\tilde x_\tau \rangle+\Exp_{\boldsymbol{\xi}}
[Q((\|\tilde x_\tau\|_2,0,\dots,0)^\top ,\boldsymbol{\xi})],\;
0\leq\tau
\leq\max\{\|x\|_2:\;Ax\leq b,\,x\geq 0\}\big\},
$$
where for each $\tau$, $\tilde x_\tau$ is the optimal solution to the convex
problem \eqref{decsptau1},
and $Q((\|\tilde x_\tau\|_2,0,\dots,0)^\top,\xi)$ 
is given by \eqref{sp_recourse} with the vector $(\|\tilde x_\tau\|_2,0,\dots,0)^\top$
in place of $x$.
Then $|z^*-\hat z^*|\leq 2\varepsilon$.
\end{prop}
\begin{proof}
Let $\tau\geq 0$ be such that
$$
\hat z^*=
\langle c,\tilde x_\tau \rangle+\Exp_{\boldsymbol{\xi}}
[Q((\|\tilde x_\tau\|_2,0,\dots,0)^\top , \boldsymbol{\xi})].
$$
By the approximate rotational invariance property,
$$
\hat z^*\leq \langle c,\tilde x_\tau \rangle+\Exp_{\boldsymbol{\xi}}
[Q(\tilde x_\tau , \boldsymbol{\xi})]+\varepsilon,
$$
and hence
$$
\hat z^*\leq z^*+\varepsilon.
$$

Conversely, let $x^*$ be an optimal solution of \eqref{sp}, so that
$$
z^*= \langle c,x^* \rangle+\Exp_{\boldsymbol{\xi}}
[Q(x^* , \boldsymbol{\xi})].
$$
Again applying the approximate rotational invariance, we get
$$
z^*\leq \langle c,x^* \rangle+\Exp_{\boldsymbol{\xi}}
[Q((\|x^*\|_2,0,\dots,0)^\top , \boldsymbol{\xi})]+\varepsilon.
$$
If $\|x^*\|_2 \leq\| \tilde x\|_2$ then, setting $\tau:=\|x^*\|_2$,
we get $\|\tilde x_\tau\|_2=\tau$ (to verify that, assume for contradiction that $\|\tilde x_\tau\|_2<\tau$ and
consider a linear interpolation between $\tilde x_\tau$ and $\tilde x$), and
hence
$$
\langle c,x^* \rangle+\Exp_{\boldsymbol{\xi}}
[Q((\|x^*\|_2,0,\dots,0)^\top , \boldsymbol{\xi})]
\leq \langle c,\tilde x_\tau \rangle+\Exp_{\boldsymbol{\xi}}
[Q((\|\tilde x_\tau\|_2,0,\dots,0)^\top , \boldsymbol{\xi})]
\leq \hat z^*,
$$
implying that
$$
z^*\leq \hat z^*+\varepsilon.
$$
If $\|x^*\|_2 >\| \tilde x\|_2$ then, by the approximate monotonicity,
$$
\langle c,x^* \rangle+\Exp_{\boldsymbol{\xi}}
[Q((\|x^*\|_2,0,\dots,0)^\top , \boldsymbol{\xi})]
\leq \langle c,\tilde x \rangle+\Exp_{\boldsymbol{\xi}}
[Q((\|\tilde x\|_2,0,\dots,0)^\top , \boldsymbol{\xi})]+\varepsilon,
$$
and
$$
z^*\leq \hat z^*+2\varepsilon.
$$
The result follows.
\end{proof}

\begin{remark}
Assume that $\{T(\xi)\}_{\xi\in\cS}$ are realizations of independent identically distributed random matrices with i.i.d centered Gaussian entries in every row. Recall that any linear combination of a set of jointly Gaussian random variables is Gaussian.
Therefore, if $x$ and $x'$ are any pair of fixed vectors in $\R^{n_1}$ having the same Euclidean norm, then the sets of random vectors
\[ \left\{ T(\xi)x \right\}_{\xi}
\ \text{and} \ \left\{ T(\xi)x' \right\}_{\xi} \]
are {\it equidistributed}. This suggests that the expected objective value of the recourse problem in this setting is influenced essentially by the norm of a vector $x$ and not by its direction. 
Therefore, we can ``replace'' any $x$ in the recourse problem by the first standard basis vector times the norm $\|x\|_2$, without significantly altering the optimal objective of the entire two-stage program. Although the above reasoning is not rigorous, our experimental results (see further) suggest that recourse problems with Gaussian technology and recourse matrices indeed have an approximate rotational invariance property.
\end{remark}

\begin{remark}
In practical implementations of the above
decoupling procedure, instead of considering a continuum
of admissible values of $\tau\in[0,
\max\{\|x\|_2:\;Ax\leq b,\,x\geq 0\}]$ we will use
a discretization $\tau=k\Delta$, $1\leq k\leq K$,
where $\Delta>0$ is sufficiently small, and
$K\in\N$ is chosen in such a way that $K\Delta\geq \|x^*\|_2$.
For each choice of $\tau=k\Delta$, we compute
the optimal solution $x_\tau^*$ to the program \eqref{decsptau1},
and then estimate $\Exp_{\boldsymbol{\xi}}[Q((\|x_\tau^*\|_2,0,\dots,0) ,\boldsymbol{\xi})]$.
\end{remark}

\section{The decoupling algorithm} 

We will assume that our stochastic program satisfies assumptions (a)--(d) from the previous section, for some small $\varepsilon>0$.
Let us start by summarizing the main points of the decoupling procedure, before presenting a formal algorithm.
The decoupling procedure for the first and second stages involves the following steps.
\begin{enumerate}
\item   {\it{} Discretizing the norm of the first stage vector $x$:}
        For a discretization step size $\Delta$, and for any integer $k$ in a certain interval,
we consider the constraint $\|x\|_2 \leq \tau$ together with the first stage constraints of \eqref{sp}, to obtain the optimization problem \eqref{decsptau1}. The discretization step ensures that the optimal objective $z^*$ of \eqref{sp} can be approximated using our method in a finite number of steps.
         
\item   {\it{} Solving the first stage problem:}
        With the additional constraint $\|x\|_2 \leq \tau$, where $\tau = k \Delta$, we solve the first stage problem above to find an optimal solution $\tilde x_\tau$. This step involves optimizing the {\it convex} nonlinear program \eqref{decsptau1} (the convexity ensures that the problem can be solved efficiently).

\item   {\it{} Fixing the first stage variable in the recourse problem:}
        In the recourse problem, we fix the variable $x$ to the optimal solution $\tilde x_\tau$ obtained at the previous step:
        \begin{equation} \label{df2}
        \begin{split}
            &\text{maximize} \  \ \Exp_{\boldsymbol{\xi}} \left [ \langle \langle q(\boldsymbol{\xi}),y(\boldsymbol{\xi}) \rangle \right ] \\
            &\text{subject to }\  W(\xi)y(\xi) \leq h(\xi) - T(\xi)\tilde x_\tau \\
            & y(\xi) \in \R^{n_2}\\
            & \xi\in \cS.
        \end{split}
        \end{equation}
\item  {\it{}  Reformulating the recourse problem:}
        The recourse problem \eqref{df2} involves $-T(\xi)\tilde x_\tau$ on the right hand side. Given the assumption of approximate rotational invariance, $T(\xi)\tilde x_\tau$ can be replaced with $T(\xi)
        \left(\|\tilde x_\tau\|_2,  0, \dots, 0 \right )^\top$
        at expense of only a small distortion of the optimal objective of the second stage. As a result, a modified problem (\ref{df2}) reads as follows 
        \begin{equation} \label{df2_2}
        \begin{split}
            &\text{maximize} \ \ \Exp_{\boldsymbol{\xi}} \left [ \langle q(\boldsymbol{\xi}),y(\boldsymbol{\xi}) \rangle \right ] \\
            &\text{subject to }\  W(\xi)y({\xi}) \leq h(\xi) - \col_1{(T(\xi))} \|\tilde x_\tau\|_2 \\
            & y(\xi) \in \R^{n_2}\\
            & \xi\in \cS.
        \end{split}
        \end{equation}
        where $\col_1{(T(\xi))}$ is the first column of the matrix $T(\xi)$.
Note that this crucial step essentially decouples the first and the second stage variables as the solution of the modified recourse problem only depends on the Euclidean norm of $\tilde x_\tau$ and not its direction.
Furthermore, since for $\tau\leq \|\tilde x\|_2$ we have $\|\tilde x_\tau\|_2=\tau$, in that range of $\tau$ the decoupling is complete, and the problems \eqref{decsptau1} and \eqref{df2_2} can be solved in parallel.
\end{enumerate}

\noindent The decoupling procedure is presented in Algorithm \ref{alg 1}. In this algorithm, $K$ is an upper bound for $k$ which is initialized with a sufficiently large number and is updated at the beginning of the algorithm according to the norm of optimal objectives of \eqref{decsptau1}. Further, $S$ is the set of scenarios, $Z_1$ and $Z_2$ are vectors for storing first and second stage optimal objectives, respectively, and $X$ is a vector indexed over $[0,\dots,K]$ that stores the norms of $\tilde x_{\Delta\,k}$, for each admissible $k$.  

\begin{algorithm}[h!]
    \SetAlgoLined
    \KwIn{ $A, W, T, b, h, c, q$; $\Delta$ (a discretization step); $K$ (a sufficiently large parameter)}
    \KwOut{ An approximation to the optimal objective value of (\ref{sp})}
    \For {$k \leftarrow 0$ to $K$}{
      Solve problem \eqref{decsptau1} with $\tau:=\Delta\,k$\;
      $Z_1[k] \leftarrow$ optimal objective value of \eqref{decsptau1}\;
      $X[k] \leftarrow$ norm of the optimal solution of \eqref{decsptau1}\;
    }
    $ a \leftarrow  \max_k X[k]$\; 
    $ K \leftarrow a/\Delta +2 $ (updating the upper bound $K$ to optimize the execution time)\;
    \For{$s \leftarrow 0$ to $S$}{
      \For{$k \leftarrow 0$ to $K$}{
        Compute the RHS of (\ref{df2_2}) taking the value of $\|\tilde x_{\Delta\,k}\|_2$ from $X[k]$\;
        Solve problem (\ref{df2_2}) \;
        $Z_2[s,k] \leftarrow $ optimal objective value of (\ref{df2_2})\;
      }
    }
    Compute $\Exp_{\boldsymbol{\xi}} Z_2[\boldsymbol{\xi}, k]$ for every $k$\;
    \For {$k \leftarrow 0$ to $K$}{
      $Z[k] \leftarrow \Exp_{\boldsymbol{\xi}} Z_2[\boldsymbol{\xi}, k] + Z_1[k]$\;
    }
    Return $\max_k Z(k)$\;

    \caption{The decoupling algorithm under approximate rotational invariance}
    \label{alg 1}
\end{algorithm}

\section{Computational Studies}\label{numericssection}

To evaluate the performance of the proposed algorithm, we conduct a series of tests with various sizes of matrices $A$, $T$ and $W$, and compare it with both the extensive form and Benders' decomposition. 
Recall that the extensive form combines all possible scenarios into a single large scale linear problem. The optimal objective of this problem serves as a benchmark to measure the accuracy of our algorithm.
Benders' decomposition, a widely used decomposition algorithm for two stage stochastic programs, is implemented as a second benchmark. We evaluate the Benders decomposition in terms of computational time.
We solve the extensive form and the subproblems of Benders' decomposition and our decoupling algorithm with help of the optimization solver Gurobi 10.1.1.

The third benchmark, which we refer to as the {\it naive decoupling}, consists in fixing the first stage variables to the optimal solution of problem (\ref{nd1}), denoted as $\tilde{x}$, and solving the recourse problems for that choice of the first-stage variables: 
\begin{equation} \label{nd2}
\begin{split}
    \text{maximize}\  & \langle c,\tilde{x} \rangle + \sum_{\xi\in\cS} \Prob(\xi)  \langle q(\xi), y(\xi) \rangle  \\
    \text{subject to } \
    & T(\xi) \tilde{x} +W(\xi) y(\xi) \leq h(\xi) \\
    & y(\xi) \geq 0.
\end{split}
\end{equation}
By comparing the optimal objective value of (\ref{nd2}) with that of the extensive form, we gain insights into the parameter configurations where our algorithm proves to be effective.

These metrics are useful for assessing the relative efficiency of our proposed algorithm.
The computational results are presented in Table \ref{tab:1}.
In the table, $m_1$ and $n_1$ are the number of rows and columns of the matrix $A$, and  $m_2$ and $n_2$ are the number of rows and columns of the matrix $W$. 
We consider constant vectors $h$ and vary the magnitude of the components for different choices of the dimension of the matrices.
{\it{}Ngap} is the relative gap in percentage between the optimal objective of problem (\ref{nd2}) and the optimal objective of the extensive form, while {\it Dgap} is the relative gap in percentage between the optimal objective obtained by our algorithm and the optimal objective of the extensive form. 
Columns $t_e$, $t_d$ and $t_b$ denote the computational times for solving the extensive form, executing our algorithm and Benders decomposition algorithm respectively.

Input data for the experiments in Table \ref{tab:1} are as follows: 
\begin{itemize}
\item $b = \mathbf{1}$, $\Delta = 0.01$, $K=100$
\item $A$ is a realization of a standard Gaussian matrix. The realization is fixed and remains the same in all experiments.
\item $c$, $q$ are realizations of Gaussian random vectors normalized to have $\|c\|_2 = 0.5$ and $\|q\|_2=1$. The realizations are fixed and remain the same in all experiments.
\item $T(\xi)$, $W(\xi)$ are realizations of independent standard Gaussian matrices.
\item Number of scenarios = sample size = 50
\item Each number in the 6th to 10th column of the table is the empirical average over 50 runs for the given parameters.
\item Compared to the description in Algorithm~\ref{alg 1}, the range of $k$ in the main cycle for the second stage problem is optimized to reduce the computational time.
\end{itemize}


From the table, we observe that Dgap is of order $2 \%$, indicating that the optimal objective of our decoupling algorithm is in close proximity to that of the extensive form solution. 
We must note that in certain parameter settings, the approximation accuracy of our algorithm matches that of the naive decoupling. For example, when the number of the first stage variables is small, both the naive approach and our decoupling algorithm yield similar objective values. Determining the range of parameters where our algorithm significantly outperforms the naive decoupling remains an open question. 

Although solving the extensive form exhibits better performance in terms of the computational time, our decoupling algorithm outperforms Benders' decomposition for some choice of parameters.
For fair comparison we run Benders' decomposition algorithm up to $2\%$ gap between the objective obtained from the master problem and the objective of current incumbent.

\bigskip

\begin{table}[h!]
\caption{ Performance of our algorithm compared with extensive form, Benders' decomposition, and naive decoupling. Columns $t_e$, $t_d$ and $t_b$ are computational times for solving the extensive form, our decoupled model and Benders decomposition algorithm respectively. }
\label{tab:1}
\centering
\begin{tabular}{cccccccccc}
$m_1$& $n_1$ & $m_2$ & $n_2$ & $h$ & Ngap & Dgap  & $t_e$ & $t_d$ & $t_b$\\
\hline \\
100 & 5 & 100 & 5   & 2 & 8.1 & 2.1 & 0.05 & 0.41 & 0.41 \\
100 & 10 & 100 & 10 & 2 & 11.7 & 2.3 & 0.18 & 0.75 & 2.1 \\
100 & 15 & 100 & 15 & 2 & 34.3 & 2.1 & 0.2 & 1.2 & 5.18\\
100 & 20 & 100 & 20 & 2 & 20.7 & 2.8 & 0.45 & 1.6 & 10.4 \\
100 & 5 & 100 & 5   & 3 & 3.3 & 1.8 & 0.06 & 0.43 & 0.24\\
100 & 10 & 100 & 10 & 3 & 5.6 & 2.0 & 0.17 & 0.78 & 1.1\\
100 & 15 & 100 & 15 & 3 & 13.6 & 2.1 & 0.21 & 1.3 & 3.2 \\
100 & 20 & 100 & 20 & 3 & 9.3 & 2.5 & 0.45 & 1.7 & 6.1\\
100 & 5 & 100 & 5   & 4 & 2.1 & 2.1 & 0.06 & 0.46 & 0.2 \\
100 & 10 & 100 & 10 & 4 & 3.3 &2.0 & 0.16 & 0.73 & 0.8 \\
100 & 15 & 100 & 15 & 4 & 7.7 & 2.2 & 0.22 & 1.2 & 1.8\\
100 & 20 & 100 & 20 & 4 & 5.6 & 2.1 & 0.41 & 1.6 & 2.4\\
100 & 5 & 100 & 5   & 5 & 1.3 & 2.1 & 0.06 & 0.48 & 0.2\\
100 & 10 & 100 & 10 & 5 & 2.2 & 1.9 & 0.17 & 0.76 & 0.6\\
100 & 15 & 100 & 15 & 5 & 5.4 & 2.2 & 0.2 & 1.1 & 1.1\\
100 & 20 & 100 & 20 & 5 & 4.1 & 2.8 & 0.4 & 1.5 & 1.5\\

\end{tabular}
\end{table}

\section{Conclusion}

In this paper, we exploit the structure of the technology matrix to effectively decouple the first and second stages of two-stage stochastic programs. Our theoretical guarantee and empirical studies, demonstrate that when the technology matrix satisfy an approximate rotational invariance property, it is possible to decouple the two stage stochastic program efficiently. The optimal objectives of our proposed algorithm are in close proximity to the optimal objectives of the extensive form solution.
Further studies could involve testing the algorithm on larger problem instances. The question of choosing parameters when our decoupling algorithm significantly outperforms the naive decoupling remains open. 
We hope that our findings will inspire the discovery of new decomposition techniques leveraging the structure of the technology matrix.




\bibliographystyle{plainnat}
\bibliography{References}

\end{document}